\documentclass[10pt,bezier]{article}
\usepackage{amsmath,amssymb,amsfonts,graphicx,caption,subcaption}
\usepackage[colorlinks,linkcolor=red,citecolor=blue]{hyperref}

\newtheorem{prethm}{{\bf Theorem}}[section]
\newenvironment{thm}{\begin{prethm}{\hspace{-0.5
em}{\bf.}}}{\end{prethm}}
\newtheorem{prepro}{{\bf Theorem}}

\newtheorem{precor}[prethm]{{\bf Corollary}}
\newenvironment{cor}{\begin{precor}{\hspace{-0.5
em}{\bf.}}}{\end{precor}}
\newtheorem{preconj}[prethm]{{\bf Conjecture}}

\newtheorem{preremark}[prethm]{{\bf Remark}}
\newenvironment{remark}{\begin{preremark}\em{\hspace{-0.5
em}{\bf.}}}{\end{preremark}}
\newtheorem{prelem}[prethm]{{\bf Lemma}}
\newenvironment{lem}{\begin{prelem}{\hspace{-0.5
em}{\bf.}}}{\end{prelem}}
\newtheorem{preque}[prethm]{{\bf Question}}

\newtheorem{preobserv}[prethm]{{\bf Observation}}

\newtheorem{predef}[prethm]{{\bf Definition}}

\newtheorem{preproposition}[prethm]{{\bf Proposition}}

\newtheorem{preproof}{{\bf Proof.}}
\newtheorem{preprooff}{{\bf Proof}}

\newenvironment{proof}[1]{\begin{preproof}{\rm
#1}\hfill{$\Box$}}{\end{preproof}}

\newtheorem{preproofs}{{\bf The second proof of }}

\newtheorem{preprooft}{{\bf Third proof of }}

\newtheorem{preproofF}{{\bf Proof of}}

\title{\bf\Large 
Factors and connected factors in tough graphs with high isolated toughness 
}
\author{{\normalsize{\sc Morteza Hasanvand${}$} }\vspace{3mm}
\\{\footnotesize{${}$\it Department of Mathematical
 Sciences, Sharif
University of Technology, Tehran, Iran}}
{\footnotesize{}}\\{\footnotesize{ $\mathsf{morteza.hasanvand@alum.sharif.edu }$ }}}

\date{}
\def\epsilon {\text{$\varepsilon$}}
\begin{document}
\maketitle
\begin{abstract}{
Let $G$ be a graph and let $f$ be a positive integer-valued function on $V(G)$. Assume that for all $S\subseteq V(G)$, $$\sum_{v\in I(G\setminus S)}f(v)(f(v)+1)\le |S|,$$ where $I(G\setminus S)$ denotes the set of isolated vertices of $G\setminus S$. In this paper, we show that if  for all $S\subseteq V(G)$, $$\omega(G\setminus S)\le \sum_{v\in S}(f(v)-1)+1,$$ and $\sum_{v\in V(G)}f(v)$ is even, then $G$ has a factor $F$ such that for each vertex $v$, $d_F(v)=f(v)$, where $\omega(G\setminus S)$ denotes the number of components of $G\setminus S$. Moreover, we show that if for all $S\subseteq V(G)$, $$\omega(G\setminus S)\le \frac{1}{4}|S|+1,$$ and  $f\ge 2$, then $G$ has a connected factor $H$ such that for each vertex $v$, $d_H(v)\in \{f(v),f(v)+1\}$.
\\
\\
\noindent {\small {\it Keywords}:
\\
Toughness;
isolated toughness;
regular factor;
connected factor;
 $f$-factor.
}} {\small
}
\end{abstract}
%
%
%
%
%
%
%
%
%
%
%
%
%
%
\section{Introduction}
In this article, all graphs have no loop, but multiple edges are allowed and a simple graph is a graph without multiple edges.
 Let $G$ be a graph. 
The vertex set and the edge set of $G$ are denoted by $V(G)$ and $E(G)$, respectively. 
We also denote by $iso(G)$, $odd(G)$, and $\omega(G)$ the number of isolated vertices of $G$, the number of components of $G$ with odd number of vertices, and the number of components of $G$, respectively.
For a vertex set $S$ of $G$, we denote by $G[S]$ the induced subgraph of $G$ with the vertex set $S$ containing
precisely those edges of $G$ whose ends lie in $S$.
The vertex set $S$ is called
 an {\it independent} set, if there is no edge
of $G$ connecting vertices in $S$.
The maximum size of all independent sets of $G$ is denoted by $\alpha(G)$.
Let $t$ be a positive real number, 
a graph $G$ is said to be
 {\it $t$-tough}, if $\omega(G\setminus S)\le \max\{1,\frac{1}{t}|S|\}$ for all $S\subseteq V(G)$. 
Furthermore, $G$ is said to be 
{\it $t$-iso-tough}, if $iso(G\setminus S)\le \frac{1}{t}|S|$ for all $S\subseteq V(G)$. 
This definition is a little different from~\cite{Ma-Yu-2007, Ma-Liu-2003} for the sake of simplicity.
Note that when $G$ is $t$-iso-tough, for each vertex $v$, the number of its neighbours must be at least $t$ and hence the conditions $d_G(v)\ge t$ and $V(G)\ge t+1$ must automatically hold.
More generally, when $t$ is a real function on $V(G)$, we say that $G$ is 
{\it $t$-iso-tough}, 
if for all $S\subseteq V(G)$, $\sum_{v\in I(G\setminus S)}t(v)\le |S|$, 
where $ I(G\setminus S)$ denotes the set of all isolated vertices of $G\setminus S$. 
We denote by $N_G(I)$ the set of all neighbours of vertices of $I$ in $G$.
For a set $A$ of integers, 
an {\it $A$-factor} is a spanning subgraph with vertex degrees in $A$. 
Let $g$ and $f$ be two integer-valued functions on $V(G)$.
A {\it $(g,f)$-factor} of $G$ is a spanning subgraph $F$ such that for each vertex $v$, $g(v)\le d_F(v)\le f(v)$.
When $g=f-1$, we call it a $\{g,g+1\}$-factor as well.
An {\it $f$-factor} of $G$ refers to a spanning subgraph $F$ such that for each vertex $v$, $d_F(v)=f(v)$.
A {\it near $f$-factor} refers to a spanning subgraph $F$ such that for each vertex $v$, $d_F(v)=f(v)$, except for at most one vertex $u$ with $d_F(u)=f(u)+1$.
Note that when the sum of all $f(v)$ taken over all vertices $v$ is even, $F$ is a near $f$-factor if and only if $F$ is an $f$-factor.
Note that several theorems in graph theory for the existence of $f$-factors can be developed to a near $f$-factor version.
 This type of factors is useful when a factor is required for extending to connected factors with bounded degrees as the proof of Theorem~\ref{thm:tree-connected;(f,f+1)-factor}. For example, see~\cite{Cai-1997, ClosedWalks}.
For convenience, we write $\min f$ for $\min\{f(v):v\in V(G)\}$ and write $\max f$ for $\max\{f(v):v\in V(G)\}$.
Let $A$ and $B$ be two disjoint vertex sets.
We denote by $\omega_{f}(G,A,B)$ the number of components $C$ of $G\setminus (A\cup B)$ satisfying $\sum_{v\in V(C)}f(v)\not \stackrel{2}{\equiv}d_G(C, B)$, 
where $d_G(C, B)$ denotes the number of edges of $G$ with one end in $V(C)$ and the other one in $B$.
Throughout this article, all variables $k$ are positive integers.
%
%
%
%
%
%
%
%
%
%
%
%

In 1947 Tutte introduced the following criterion for the existence of a perfect matching.
\begin{thm}{\rm (\cite{Tutte-1947})}\label{thm:Tutte1-factor}
A graph $G$ has a $1$-factor if and only if for all $S\subseteq V(G)$, 
$odd(G\setminus S)\le |S|$.
\end{thm}

In 1978 Vergenas formulated a criterion for the existence of $(1,f)$-factors and showed that the criterion becomes simpler for the following special case.
\begin{thm}{\rm (\cite{LasVergnas-1978})}\label{thm:Vergenas}
{Let $G$ be a graph and let $f$ be an integer-valued function on $V(G)$ with $f\ge 2$.
Then $G$ has a $(1,f)$-factor if and only if for all $S\subseteq V(G)$,
$iso(G\setminus S)\le \sum_{v\in S} f(v).$
}\end{thm}

In 1985 Enomoto, Jackson, Katerinis, and Saito proved the following theorem on tough graphs, 
which was originally conjectured by Chv\'atal (1973)~\cite{Chvatal-1973}. In 1990 Katerinis~\cite{Katerinis-1990} generalized their result by replacing a weaker sufficient toughness condition for the existence of $[a,b]$-factors, provided that $a > b$.
\begin{thm}{\rm (\cite{Enomoto-Jackson-Katerinis-Saito-1985})}\label{thm:tough:r-factor}
Every $k$-tough graph $G$ of order at least $k+1$
with $k|V(G)|$ even has a $k$-factor.
\end{thm}

In 2007 Ma and Yu strengthened Katerinis' result by replacing isolated toughness condition as the following theorem.
In this paper, we provide a supplement for their result by improving Theorem~\ref{thm:tough:r-factor} 
for $(k+1)$-iso-tough graphs by showing that the needed toughness can be pushed down to the fixed number $1$.
In Section~\ref{sec:Graphs with higher toughness}, we also establish another refined  version in $k$-iso-tough graphs.
\begin{thm}{\rm (\cite{Ma-Yu-2007})}\label{thm:intro:iso}
Every $(a-\frac{b-a}{b})$-iso-tough graph 
has an $[a, b]$-factor, when $b>a\ge 1$.
\end{thm}

In 1973 Chv\'atal~\cite{Chvatal-1973} conjectured that there exists a positive real number $t_0$ such that every $t_0$-tough graph of order at least three admits a Hamiltonian cycle. 
In 2000 Ellingham and Zha~\cite{Ellingham-Zha-2000} confirmed a weaker version of this conjecture by proving that every
$4$-tough graph of order at least three admits a connected $\{2,3\}$-factor.
Motivated by this result,  one way ask whether higher toughness can guarantee the existence of connected $\{k,k+1\}$-factors.
The following theorem shows that the answer is positive. In this paper, we show that the needed toughness of this theorem can be pushed down to the fixed number $3$ but in $(k+1)$-iso-tough graphs.
\begin{thm}{\rm (\cite{Ellingham-Nam-Voss-2002,Ellingham-Zha-2000}, see \cite{ClosedWalks})}
Every $k$-tough graph of order at least $k+1$ has a connected $\{k,k+1\}$-factor, where $k\ge 3$.
\end{thm}

In 1990 Katerinis formulated the following sufficient toughness condition for the existence of $f$-factors. In Section~\ref{sec:f-factors}, we introduce some sufficient toughness conditions for the existence of $f$-factors and connected $\{f,f+1\}$-factors in graphs with high enough isolated toughness as mentioned in the abstract. 
\begin{thm}{\rm (\cite{Katerinis-1990})}
Let $G$ be a graph and let $f$ be a positive integer-valued function on $V(G)$ satisfying $a\le f\le b$, where $a$ and $b$ are two positive integers.
If $G$ is $\frac{1}{4a}((b+a)^2+2(b-a)+1)$-tough and $\sum_{v\in V(G)}f(v)$ is even, then $G$ has an $f$-factor.
\end{thm}
\section{Tools and preliminary results}
In this section, we shall provide some necessary tools for applying in the next sections.
Before doing so, let us recall a theorem due to Tutte (1952) as the following version.
\begin{thm}{\rm (\cite{Tutte-1952})}\label{thm:base}
{Let $G$ be a general graph and let $f$ be an integer-valued function on $V(G)$. 
Then $G$ has a near $f$-factor if and only if
for all disjoint subsets $A$ and $B$ of $V(G)$,
$$\omega_f(G, A,B)\le \sum_{v\in A} f(v)+\sum_{v\in B} (d_{G\setminus A}(v)-f(v))+1.$$
}\end{thm}
The following corollary is an application of Theorem~\ref{thm:base}, which is inspired by Lemma 4 in~\cite{Katerinis-1990}.
\begin{cor}{\rm (see \cite{Katerinis-1990})}\label{cor:main}
{Let $G$ be a general graph and let $f$ be an integer-valued function on $V(G)$. 
Then $G$ has a near $f$-factor, if 
$$\omega (G\setminus (A\cup B))\le \sum_{v\in A} f(v)+\sum_{v\in B} (d_{G\setminus A}(v)-f(v))+1,$$
for all disjoint subsets $A$ and $B$ of $V(G)$ satisfying 
$d_{G [B]}(u) \le f(u)-2$ and $d_{G\setminus A}(u) \le 2f(u)-1$ for each $u\in B$.
}\end{cor}
\begin{proof}
{Let us define $g=f$. We are going to show that the inequality holds for any two disjoint subsets $A$ and $B$ of $V(G)$ and so the proof follows from Theorem~\ref{thm:base}. 
 By induction on $|B|$.
Let $q(A,B)$ be the right-hand side of the inequality in the corollary.
Assume that $B$ has a vertex $u$ with 
$ d_{G[B]}(u)\ge g(u)-1$ 
or 
$d_{G\setminus A}(u)\ge f(u)+g(u)$.
Define $B_u=B\setminus \{u\}$. 
If $d_{G[B]}(u) \ge g(u)-1$, then
$$\omega(G\setminus (A\cup B))\le \omega(G\setminus (A\cup B_u))+d-1 \le 
q(A,B_u)+d-1 =q(A,B)-d_{G[B]}(u)+g(u)-1 \le
 q(A,B).$$
where $d$ denotes the number of edges of $G$ incident to $u$ with the other end in $V(G)\setminus (A\cup B)$.
Also, if $d_{G\setminus A}(u) \ge f(u)+g(u)$, then 
$$\omega(G\setminus (A\cup B))=
 \omega(G\setminus (A_u\cup B_u)) \le 
q(A_u,B_u)= 
q(A,B)-d_{G\setminus A}(u) +f(u)+g(u) \le q(A,B),$$
where $A_u=A\cup \{u\}$. Hence the lemma holds. 
}\end{proof}
The following theorem can generalize Lemma~1 in~\cite{Katerinis-1990} and plays an essential role in this paper.
\begin{thm}\label{thm:base:independent-subset}
{Let $H$ be a graph. If $\varphi$ is a nonnegative real function on $V(H)$,
 then there is a maximal independent subset $I$ of $V(H)$ such that 
$$
\sum_{v\in V(H)} \varphi(v)\le \sum_{v\in I}\varphi(v) (d_H(v)+1) .
$$
}\end{thm}
\begin{proof}
{Define $H_0=H$. 
For every nonnegative integer $i$ with $|V(H_i)|\ge 1$, take $u_i$ to be a vertex of $H_i$ with the maximum $\varphi (u_i)$ and
 set $H_{i+1}=H_{i}\setminus (N(u_i)\cup \{u_i\})$, where $N(u_i)$ denotes the set of all neighbours of $u_i$ in $H_i$.
Define $I$ to be the set of all selected vertices $u_i$.
It is not hard to check that $I$ is a maximal independent set of $H$ and 
$\{ N(u) \cup \{u\}: u\in I \}$
is a partition of $V(H)$. Since $\varphi (u)\ge 0$, 
$$
\sum_{v\in V(H)\setminus I}\varphi(v) = 
\sum_{u\in I}\sum_{v\in N(u)} \varphi(v) \le 
\sum_{u\in I} \varphi(u)d_{H}(u).
$$
This inequality can complete the proof. 
}\end{proof}
\begin{cor}\label{lem:independence}\label{lem:optimized}{\rm (\cite{Caro-1979,Wei-1981})}
{For every graph $H$, we have 
$\alpha(H)\ge \sum_{v\in V(H)}\frac{1}{1+d_H(v)}.$
}\end{cor}
\begin{proof}
{Apply Theorem~\ref{thm:base:independent-subset} with $\varphi(v) = 1/(1+d_H(v))$.
}\end{proof}
The following corollary provides an equivalent version for Theorem~\ref{thm:base:independent-subset}.
\begin{cor}\label{cor:base}
{Let $H$ be a graph and let $\varphi$ and $d$ be two real functions on $V(H)$. 
If for each $v\in V(H)$, $\varphi (v) \ge d(v) \ge d_H(v)$, then there is a maximal independent subset $I$ of $V(H)$ such that 
$$
\sum_{v\in V(H)} (\varphi(v)-d(v)) \le \sum_{v\in I}(d(v)+1) (\varphi(v)-d(v)).
$$}\end{cor}
\begin{proof}
{Apply Theorem~\ref{thm:base:independent-subset} with replacing $(\varphi(v)-d(v))$ instead of $\varphi(v)$.
}\end{proof}

\section{Isolated toughness and the existence of $\{f,f+1\}$-factors}
Our aim in this section is to generalize Theorem~\ref{thm:intro:iso} by giving isolated toughness conditions for existence of $(g,f)$-factors, provided that $g<f$.
For this purpose, we need the following lemma due to Lov\'{a}sz (1970).
\begin{lem}{\rm (\cite{Lovasz-1970})}\label{lem:Lovasz}
{Let $G$ be a graph and let $g$ and $f$ be two integer-valued functions on $V(G)$ with $g< f$. 
Then $G$ has a $(g,f)$-factor, if and only if 
for all disjoint subsets $A$ and $B$ of $V(G)$,
$$0\le \sum_{v\in A} f(v)+\sum_{v\in B} (d_{G\setminus A}(v)-g(v)).$$
}\end{lem}
The following theorem provides a common generalization for both of Theorems~2 and~3 in~\cite{Ma-Yu-2007}.
\begin{thm}\label{thm:(g,f)-factor}
{Let $G$ be a graph and let $g$ and $f$ be two nonnegative integer-valued functions on $V(G)$ with $g<f$.
Let $a$ be positive real number with $a\le f$.
Then  $G$ has a  $(g,f)$-factor, if it $t$-iso-tough, where for each vertex $v$,
$$t(v)= 
 \begin{cases}
g(v)(1+\frac{1}{a})-1,	&\text{when $g(v)\le a+2$};\\ 
\frac{1}{4a}((g+a+1)^2-\epsilon_0(v))-1,	&\text{otherwise},
\end {cases}$$
 in which $\epsilon_0(v)\in \{0,1\}$ such that $\epsilon_0(v)=1$ 
if and only if $g(v)$ and  $a$ are integers with the same parity.
}\end{thm}
\begin{proof}
{Let $A$ and $B$ be two disjoint subsets of $V(G)$.
In order to apply Lemma~\ref{lem:Lovasz}, we should prove the inequality 
$0\le \sum_{v\in A} f(v) +\sum_{v\in B}(d(v)-g(v))$, where $d(v)=d_{G\setminus A}(v)$.
For this purpose, we may assume that for each $v\in B$, $d(v)\le g(v)-1$. 
By applying Corollary~\ref{cor:base} with $\varphi=g$, the graph $G[B]$ has an independent set $I$ such that 
\begin{equation*}
\sum_{v\in B} (g(v)-d(v)) \le 
 \sum_{v\in I}(d(v)+1) (g(v)-d(v)).
\end{equation*}
Since $G$ is $t$-iso-tough, we have 
$\sum_{v\in I}t(v)\le |A\cup N_G(I)| \le |A|+\sum_{v\in I}d(v)$.
Since $d(v)$ is integer, we must have 
$$(d(v)+1) \big(g(v)-d(v)\big)+a\, d(v)=(d(v)+1) \big(g(v)+a-d(v)\big) -a\le a\, t(v),$$
regardless of $g(v)-1\le (g(v)+a)/2$ or not.
This implies that 
\begin{equation*}
\sum_{v\in B} (g(v)-d(v)) \le \sum_{v\in I}(d(v)+1) \big(g(v)-d(v)\big)\le 
\sum_{v\in I}\big(a \, t(v) -a\, d(v)\big) \le a |A|.
\end{equation*}
Therefore, 
\begin{equation*}
0 \le a|A|+\sum_{v\in B}(d(v)-g(v)) \le \sum_{v\in A} f(v)+\sum_{v\in B}(d(v)-g(v)).
\end{equation*}
Hence the assertion follows from Lemma~\ref{lem:Lovasz}.
}\end{proof}
When we consider the special case $\max g< \min f$, the theorem becomes simpler as the following result.
\begin{cor}
{Let $G$ be a graph and let $g$ and $f$ be two nonnegative integer-valued functions on $V(G)$ with $\max g< \min f$.
If $G$ is $(g-1+\frac{g}{\min f})$-iso-tough,
then it has a $(g,f)$-factor.
}\end{cor}
\begin{proof}
{Apply Theorem~\ref{thm:(g,f)-factor} with $a=\min f$.
}\end{proof}
\begin{cor}
{Every $f(f+1)$-iso-tough graph $G$ has an $\{f,f+1\}$-factor, where $f$ is a nonnegative integer-valued function on $V(G)$.
}\end{cor}
\begin{proof}
{Apply Theorem~\ref{thm:(g,f)-factor} with setting and $g'=f$, $f'=f+1$, and $a=1$.
}\end{proof}
\section{Toughness, isolated toughness, and the existence of $f$-factors}
\label{sec:f-factors}
In this section, we are going to present some sufficient toughness conditions for the existence of $f$-factors in graphs with high enough isolated toughness. 
\subsection{Regular factors in $1$-tough graphs}
The following theorem significantly improves
 the needed toughness in Theorem~\ref{thm:tough:r-factor}  in graphs with a bit higher isolated toughness.
\begin{thm}\label{thm:main:theorem}
{Let $k$ be a positive integer and let $n$ be a real number with $n\ge 1$. 
If  $G$ is  a $(k+1/n)$-iso-tough graph and 
for all $S\subseteq V(G)$, $$\omega(G\setminus S)< \frac{1}{n}|S|+2,$$
then $G$ has a near $k$-factor.
}\end{thm}
\begin{proof}
{Let $A$ and $B$ be two disjoint subsets of $V(G)$ such that for each $v\in B$, $d_{G[B]}(v)\le k-2$.
By applying a greedy coloring, one can decompose $B$ into $k-1$ independent vertex sets $B_1,\ldots, B_{k-1}$.
Let $S_i=A\cup N_G(B_i)$. 
By the assumption, we must have 
$$(k+\frac{1}{n})
 |B_i| \le
 (k+\frac{1}{n})
\; iso(G\setminus S_i) \le |S_i|\le |A|+ \sum_{v\in B_i}d_{G\setminus A}(v),$$
which implies that
$$(k+\frac{1}{n})|B|=
\sum_{1\le i < k}
(k+\frac{1}{n})
|B_i| \le 
(k-1)|A|+\sum_{1\le i < k} \sum_{v\in B_i}d_{G\setminus A}(v) = 
(k-1)|A|+\sum_{v\in B}d_{G\setminus A}(v).$$
By the assumption, we must also have 
$$\omega(G\setminus A\cup B) < 
\frac{1}{n}
(|A|+|B|)+2.$$
 Therefore, 
$$\omega(G\setminus (A\cup B)) < 
(k-1+\frac{1}{n}
)
|A|+\sum_{v\in B} (d_{G\setminus A}(v)-k)+2\le
k|A|+\sum_{v\in B} (d_{G\setminus A}(v)-k)+2.$$
Thus the assertion follows from Corollary~\ref{cor:main}. 
}\end{proof}
\begin{remark}
{Note that when $G$ has no complete subgraphs of order $k-1$ and $k\ge 5$,  independent sets $B_i$ could be chosen such that $B_{k-1}=\emptyset$ using Brooks' Theorem~\cite{Brooks-1941}. This fact allows us to reduce the lower bound on $n$ to $1/2$.
}\end{remark}
\subsection{Graphs with toughness less than $1$}
The following theorem gives a sufficient toughness condition for the existence of  $f$-factors.
\begin{thm}\label{thm:iso:f-factor}
{Let $\epsilon$ be a real number with $0<\epsilon \le 1$. 
Let $G$ be a graph and let $f$ be  a positive integer-valued function on $V(G)$.
If $G$ is $f(f+1)/\epsilon$-iso-tough and for all $S\subseteq V(G)$, 
$$\omega(G\setminus S)< \sum_{v\in S}(f(v)-\epsilon)+2,$$
then $G$ has a near $f$-factor.
}\end{thm}
\begin{proof}
{Let $A$ and $B$ be two disjoint subsets of $V(G)$.
We may assume that $d(v)\le 2f(v)-1$ for each $v\in B$, where $d(v)= d_{G\setminus A}(v)$.
For each $v\in B$, define $\varphi(v) =2f(v)-\epsilon$ so that $\varphi(v)\ge d(v)$.
By Corollary~\ref{cor:base}, the graph $G[B]$ has an independent set $I$ such that 
\begin{equation*}
\sum_{v\in B} (\varphi(v)-d(v)) \le 
 \sum_{v\in I}(d(v)+1) (\varphi(v)-d(v)).
\end{equation*}
For each vertex $v$, define $t(v)=f(v)(f(v)+1)/\epsilon-1$.
Since $G$ is $t$-iso-tough, we have 
$\sum_{v\in I}t(v)\le |A\cup N_G(I)| \le |A|+\sum_{v\in I}d(v)$.
Since $d(v)$ is integer, we must have 
$$(d(v)+1) \big(\varphi(v)-d(v)\big) +\epsilon\, d(v)=
(d(v)+1) \big(2f(v)-d(v)\big)-\epsilon \le
 \epsilon \, t(v),$$
which implies that 
\begin{equation*}
\sum_{v\in B} (\varphi(v)-d(v)) \le \sum_{v\in I}(d(v)+1) \big(\varphi(v)-d(v)\big)\le 
\sum_{v\in I}\big(\epsilon \, t(v) -\epsilon d(v)\big) \le \epsilon |A|.
\end{equation*}
On the other hand, by the assumption, 
$$\omega(G\setminus (A\cup B)) < \sum_{v\in A\cup B}(f(v)-\epsilon)+2=\sum_{v\in A}(f(v)-\epsilon)+\sum_{v\in B}(f(v)-\epsilon)+2.$$
Therefore, 
\begin{equation*}
\omega(G\setminus (A\cup B)) \le \sum_{v\in A} f(v)+\sum_{v\in B}(d(v)-f(v))+1.
\end{equation*}
Hence the assertion follows from Corollary~\ref{cor:main}.
}\end{proof}
\begin{cor}\label{cor:iso:epsilon:f-1}
{Let $G$ be a graph and let $f$ be a positive integer-valued function on $V(G)$. If $G$ is $f(f+1)$-iso-tough and
 for all $S\subseteq V(G)$, 
$$\omega(G\setminus S)\le \sum_{v\in S}(f(v)-1)+1,$$
then $G$ has a near $f$-factor.
}\end{cor}
\begin{proof}
{Apply Theorem~\ref{thm:iso:f-factor} with $\epsilon=1$.
}\end{proof}
The isolated toughness needed in Corollary~\ref{cor:iso:epsilon:f-1} can be improved by a coefficient for graphs with higher toughness as the next theorem, provided that $\min f$ is sufficiently large.
\begin{thm}\label{thm:toughness:less:one}
{Let $G$ be a graph and let $f$ be a positive integer-valued function on $V(G)$, and
let $a$ be a positive real number with $f\ge a\ge 1$.
If $G$ is $\frac{1}{a}(f+a/2)^2$-iso-tough and 
 for all $S\subseteq V(G)$, 
$$\omega(G\setminus S)< \sum_{v\in S}(f(v)-a)+2,$$
then $G$ has a near $f$-factor.
}\end{thm}
\begin{proof}
{Let $A$ and $B$ be two disjoint subsets of $V(G)$.
We may assume that $d(v)\le 2f(v)-1$ for each $v\in B$, where $d(v)= d_{G\setminus A}(v)$.
For each $v\in B$, define $\varphi(v) =2f(v)-1$ so that $\varphi(v)\ge d(v)$.
By Corollary~\ref{cor:base}, the graph $G[B]$ has an independent set $I$ such that 
\begin{equation*}
\sum_{v\in B} (\varphi(v)-d(v)) \le 
 \sum_{v\in I}(d(v)+1) (\varphi(v)-d(v)).
\end{equation*}
For each vertex $v$, define $t(v)=\frac{1}{a}((f(v)+a/2)^2)-1$.
Since $G$ is $t$-iso-tough, we have 
$\sum_{v\in I}t(v)\le |A\cup N_G(I)| \le |A|+\sum_{v\in I}d(v)$.
In addition, we must have 
$$  (d(v)+1) \big(\varphi(v)-d(v)\big)+a\, d(v)=(d(v)+1) \big(2f(v)-1+a-d(v)\big) -a\le a\, t(v),$$
which implies that 
\begin{equation*}
\sum_{v\in B} (\varphi(v)-d(v)) \le \sum_{v\in I}(d(v)+1) \big(\varphi(v)-d(v)\big)\le 
\sum_{v\in I}\big(a \, t(v) -a d(v)\big) \le a |A|.
\end{equation*}
On the other hand, by the assumption, 
$$\omega(G\setminus (A\cup B)) < \sum_{v\in A\cup B}(f(v)-a)+2=\sum_{v\in A}(f(v)-a)+\sum_{v\in B}(f(v)-a)+2.$$
Therefore, 
\begin{equation*}
\omega(G\setminus (A\cup B)) \le \sum_{v\in A} f(v)+\sum_{v\in B}(d(v)-f(v))+1.
\end{equation*}
Hence the assertion follows from Corollary~\ref{cor:main}.
}\end{proof}
\subsection{Applications to the existence of connected $\{f,f+1\}$-factors}
\label{subsec:connected-factor}
The following lemma is a useful tool for extending factors to connected factors by inserting a matching.
\begin{lem}\label{lem:tough:1-Extend}{\rm (\cite{Ellingham-Zha-2000}, see \cite{ClosedWalks})}
{Let $\epsilon$ be a real number with $0< \epsilon \le 2$. Let $G$ be a simple graph and let $F$ be a factor of $G$ with minimum degree at least $2/\epsilon+1$.
If for all $S\subseteq V(G)$,
$$\omega(G\setminus S) \le 
\frac{1}{2+\epsilon }|S|+1,$$
then $G$ has a connected factor $H$ containing $F$ 
such that for each vertex $v$, 
 $d_H(v) \in \{d_F(v), d_F(v)+1\}$, and $d_H(u)=d_F(u)$ for an arbitrary given vertex $u$.
}\end{lem}
The following result shows an application of Lemma~\ref{lem:tough:1-Extend} and Theorem~\ref{thm:main:theorem}.
\begin{thm}
Every $3$-tough $(k+1/3)$-iso-tough graph has a connected $\{k,k+1\}$-factor, where $k\ge 3$.
\end{thm}
\begin{proof}
{We may assume that $G$ simple, by deleting multiple edges from $G$ (if necessary). 
By Theorem~\ref{thm:main:theorem}, the graph $G$ has a near $k$-factor $F$ so that for all vertices $v$, $d_F(v)=k$, 
 except for at most one vertex $u$ with  $d_F(u)=k+1$.
By applying Lemma~\ref{lem:tough:1-Extend} with $\epsilon=1$, 
the graph $G$ has a connected factor $H$
such that for each vertex $v$,  $d_H(v) \in \{d_F(v), d_F(v)+1\}$, and also $d_H(u)=d_F(u)$.
This implies that $H$ is a connected $\{k,k+1\}$-factor. 
}\end{proof}
The next result shows an application of Lemma~\ref{lem:tough:1-Extend} and Corollary~\ref{cor:iso:epsilon:f-1}.
\begin{thm}\label{thm:tree-connected;(f,f+1)-factor}
{Let $\epsilon$ be a real number with $0< \epsilon \le 2$. 
 Let $G$ be graph and let $f$ be a positive integer-valued function on $V(G)$ with $f\ge 2/\epsilon+1$.
If $G$ is $f(f +1)$-iso-tough and for all $S\subseteq V(G)$, 
$$\omega(G\setminus S)\le \frac{1}{2+\epsilon}|S|+1,$$ 
then $G$ has a connected $\{f,f+1\}$-factor.
}\end{thm}
\begin{proof}
{We may assume that $G$ simple, by deleting multiple edges from $G$ (if necessary). 
By Corollary~\ref{cor:iso:epsilon:f-1}, the graph $G$ has a near $f$-factor $F$ so that for all vertices $v$, $d_F(v)=f(v)$, 
 except for at most one vertex $u$ with $d_F(u)=f(u)+1$. By applying Lemma~\ref{lem:tough:1-Extend}, the graph $G$ has a connected factor $H$ such that for each vertex $v$,  $d_H(v) \in \{d_F(v), d_F(v)+1\}$, and also $d_H(u)=d_F(u)$. This implies that $H$ is a connected $\{f,f+1\}$-factor. 
}\end{proof}
\begin{cor}
{Every $3$-tough $f(f +1)$-iso-tough graph $G$ has a connected $\{f,f+1\}$-factor, where $f$ is a positive integer-valued function on $V(G)$ with $f\ge 3$.
}\end{cor}
\begin{proof}
{Apply Theorem~\ref{thm:tree-connected;(f,f+1)-factor} with $\epsilon=1$.
}\end{proof}
%
%
%
%
%
%
%
%
%
%
%
%
%
\section{Graphs with higher toughness}
\label{sec:Graphs with higher toughness}
Our in this section is to provide  another improvement for Theorem~\ref{thm:tough:r-factor}.
Before doing so, let us  refine Theorem~\ref{thm:toughness:less:one} slightly for graphs with higher toughness.
\begin{thm}\label{thm:iso:f-factor:second-method}
{Let $G$ be a graph, let $f$ be a positive integer-valued function on $V(G)$, and let $a$ be 
a real number with $f\ge a > 1$.
For each vertex $v$, let $\epsilon_0(v)\in \{0,1\}$ such that $\epsilon_0(v)=1$ if only if  $f(v)$ and $a$ are integers with the same parity.
Then $G$ has a near $f$-factor, if 
 $G$ is $\frac{1}{4(a-1)}((f(v)+a-1)^2-\epsilon_0(v))$-iso-tough and for all $S\subseteq V(G)$,
$$\omega(G\setminus S)+\sum_{v\in I_*(G\setminus S)}(f(v)-1)\le |S|+1,$$
where $I_*(G\setminus S)$ is  the set of  center vertices of 
the star components of $G\setminus S$ in which for stars with one edge $xy$ the vertex $x$  is center whenever $f(x)\ge f(y)$.
}\end{thm}
\begin{proof}
{The proof presented here is inspired by the proof of Theorem~1 in~\cite{Katerinis-1990}. 
Let $A$ and $B$ be two disjoint subsets of $V(G)$.
For notational
simplicity, we write $d(v)$ for $d_{G\setminus A}(v)$.
Define $B_0=\{v\in B:d(v) < f(v)\}$.
By Corollary~\ref{cor:base}, the graph $G[B_0]$ has an independent set $I_0$ such that 
\begin{equation}\label{1-tough:eq:1:iso:f-factor}
\sum_{v\in B_0} (f (v)-d(v)) \le 
 \sum_{v\in I_0}(d(v) +1)(f(v)-d(v)).
\end{equation}
For each vertex $v$, let $t(v)=\frac{1}{4(a-1)}((f(v)+a-1)^2-\epsilon_0(v))$.
Since $G$ is $t$-iso-tough, we have 
$\sum_{v\in I_0}t(v)\le |A\cup N_G(I_0)| \le |A|+\sum_{v\in I_0}d(v)$. 
Since $d(v)$ is integer, we must have
$$ (d(v)+1) \big(f(v)-d(v)\big) +a\, d(v)-f(v)=d(v)(f(v)+a-1-d(v))\le (a-1)t(v),$$
which implies that 
\begin{equation}\label{1-tough:eq:3:iso:f-factor}
\sum_{v\in I_0}(d(v)+1) \big(f(v)-d(v)\big)-\sum_{v\in I_0}f(v) \le 
\sum_{v\in I_0}\big((a-1) \, t(v) -a \, d(v)\big) \le (a-1) |A|-\sum_{v\in I_0} d(v) .
\end{equation}
Therefore, 
Relations~(\ref{1-tough:eq:1:iso:f-factor}) and~(\ref{1-tough:eq:3:iso:f-factor}) 
can deduce that
\begin{equation}\label{1-tough:eq:4:iso:f-factor}
\sum_{v\in B_0}\big(f(v)-d(v)\big)
 \le  (a-1) |A|-\sum_{v\in I_0}(d(v)-f(v)).
\end{equation}
Let $I$ be a maximal independent set in $G[B]$ containing the vertices of $I_0$ so that 
 $B\setminus I \subseteq N_G(I)$.
Denote by $x_1$ the number of components $C$ of $G\setminus (A\cup B)$
such that $d_G(v,I)= 1$ for each $v\in V(C)$.
For every such a component $C$, select a vertex $z$ with $d_G(z, I)=1$.
Define $Z$ to be the set of all selected vertices. 
Also, denote by $x_2$ the number of components $C$ of $G\setminus (A\cup B)$
such that $d_G(v,I)\ge 1$ for each $v\in V(C)$, and $d_G(u,I)\ge 2$ for at least one vertex $u\in V(C)$.
Set $S=A\cup (N_G(I)\setminus Z)$.
According to this definition, it is not difficult to show that
$$|S| \le |A|+\sum_{v\in I}d(v)-x_1-x_2,$$
and
$$\omega(G\setminus (A\cup B))\le \omega(G\setminus S)+x_1+x_2- |I| .$$
On the other hand, by the assumption,
 $$\omega(G\setminus S) + \sum_{v\in I}(f(v)-1)
\le |S|+1,$$
which implies that 
$$\omega(G\setminus (A\cup B))\le |A|+\sum_{v\in I}(d(v)-f(v))+1.$$
Since  $d(v) \ge f(v)$ for each $v\in B\setminus (B_0\cup I)$, 
we must have 
\begin{equation}\label{1-tough:eq:5:iso:f-factor}
\omega(G\setminus (A\cup B))
\le 
|A|+\sum_{v\in I_0}(d(v)-f(v))+\sum_{v\in B\setminus B_0}(d(v)-f(v))+1.
\end{equation}
Therefore, Relations~(\ref{1-tough:eq:4:iso:f-factor}) and~(\ref{1-tough:eq:5:iso:f-factor}) can conclude that 
$$
\omega(G\setminus (A\cup B)) \le a|A|+ \sum_{v\in B}(d(v)-f(v))+1
\le \sum_{v\in A} f(v)+ \sum_{v\in B}(d(v)-f(v))+1.
$$
Hence the assertion follows from Corollary~\ref{cor:main}.
}\end{proof}
The following corollary is an improved version of Theorem~\ref{thm:tough:r-factor}.
\begin{cor}\label{cor:r-tough}
{A graph $G$ has a near $k$-factor, if
 for all $S\subseteq V(G)$, $ iso(G\setminus S)\le \frac{1}{k}|S|$, and
$$\omega(G\setminus S)+(k-1)\, w_*(G\setminus S)\le |S|+1,$$
where $w_*(G\setminus S)$ denotes the number of star components of $G\setminus S$.
}\end{cor}
\begin{proof}
{Apply Theorem~\ref{thm:iso:f-factor:second-method} with setting $f(v)=a=k$  when $k> 1$.
For the special case $k=1$, one can apply Theorem~\ref{thm:Tutte1-factor} directly.
}\end{proof}
\begin{cor}{\rm (\cite{Enomoto-Jackson-Katerinis-Saito-1985})}
{Every $k$-tough graph $G$ of order at least $k+1$ has a near $k$-factor.
}\end{cor}
\begin{proof}
{We may assume that $|V(G)|\ge k+2$. Let $S$ be a subset of $ V(G)$. 
If $|S|< k$, then $ \omega(G\setminus S)= 1$ and also $iso(G\setminus S)=0$. 
Since $|V(G)|\ge k+2$, by the assumption, 
one can conclude that each vertex of $G$ contains at least $k+1$ neighbours. Hence $w_*(G\setminus S)=0$.
If $|S|\ge k$, then we have $iso(G\setminus S) \le |S|/k$ and $\omega(G\setminus S)+(k-1)\, w_*(G\setminus S) \le k\, \omega(G\setminus S)\le |S|$. 
Now, it is enough to apply Corollary~\ref{cor:r-tough}.
}\end{proof}
%
%
%
%
%
%
%
%
%
%


\begin{thebibliography}{10}

\bibitem{Akiyama-Kano-2011}
 J.~Akiyama and M.~Kano, Factors and factorizations of graphs,
Springer, Heidelberg, 2011.

\bibitem{Brooks-1941}
 R.L. Brooks, On colouring the nodes of a network, Proc. Cambridge Philos. Soc. 37 (1941)~194--197.

\bibitem{Cai-1997}
M.-c. Cai, Connected {$[k,k+1]$}-factors of graphs, Discrete Math. 169 (1997)~1--16.

\bibitem{Caro-1979}
Y. Caro, New results on the independence number, Technical Report, Tel-Aviv University (1979).

\bibitem{Chvatal-1973}
 V.~Chv\'atal, Tough graphs and {H}amiltonian circuits, Discrete Math. 5 (1973)~215--228.

\bibitem{Ellingham-Nam-Voss-2002}
 M.N. Ellingham, Y.~Nam, and H.-J. Voss, Connected {$(g,f)$}-factors, J. Graph Theory 39 (2002)~62--75.

\bibitem{Ellingham-Zha-2000}
 M.N. Ellingham and X.~Zha, Toughness, trees, and walks, J. Graph Theory 33 (2000)~125--137.

\bibitem{Enomoto-Jackson-Katerinis-Saito-1985}
 H.~Enomoto, B.~Jackson, P.~Katerinis, and A.~Saito, Toughness and the existence of {$k$}-factors, J. Graph Theory 9 (1985)~87--95.

\bibitem{ClosedWalks}
M. Hasanvand, Spanning trees and spanning closed walks with small degrees, arXiv:1702.06203v6.

\bibitem{Katerinis-1990}
P.~Katerinis, Toughness of graphs and the existence of factors, Discrete Math. 80 (1990)~81--92.

\bibitem{LasVergnas-1978}
 M.~Las~Vergnas, An extension of {T}utte's {$1$}-factor theorem, Discrete Math. 23 (1978)~241--255.

\bibitem{Lovasz-1970}
L.~Lov\'{a}sz, Subgraphs with prescribed valencies, J. Combinatorial Theory 8 (1970)~391--416.

\bibitem{Ma-Yu-2007}
 Y.~Ma and Q.~Yu, Isolated toughness and existence of {$f$}-factors, in Discrete geometry, combinatorics and graph theory, vol.~4381 of Lecture Notes in Comput. Sci., Springer, Berlin, 2007, pp.~120--129.

\bibitem{Ma-Liu-2003}
Y.H. Ma and G.Z. Liu, Isolated toughness and the existence of fractional factors, Acta Math. Appl. Sin. (Chinese) 26 (2003)~133--140.

\bibitem{Wei-1981}
V.K. Wei, A lower bound on the stability number of a simple graph, Technical Memorandum, TM 81-11217-9, Bell laboratories (1981).

\bibitem{Tutte-1947}
 W.T. Tutte, The factorization of linear graphs, J. London Math. Soc. 22 (1947)~107--111.

\bibitem{Tutte-1952}
 W.T. Tutte, The factors of graphs, Canadian J. Math. 4 (1952)~314--328.

\end{thebibliography}
\end{document}